\documentclass[a4paper, 12pt]{article}
\usepackage[english]{babel}
\usepackage{amsmath,amssymb,amsthm,xcolor,latexsym,tabularx,stmaryrd}
\usepackage{hyperref}
\usepackage{comment}
\usepackage{a4wide}
\newcommand{\assign}{:=}
\newcommand{\mathd}{\mathrm{d}}

\newcommand{\tmop}[1]{\ensuremath{\operatorname{#1}}}

\newcommand{\NP}[1]{{\color{red}NP: #1}}

\newtheorem{theorem}{Theorem}[section]
\newtheorem*{theorem*}{Theorem}
\newtheorem{corollary}[theorem]{Corollary}
\newtheorem{lemma}[theorem]{Lemma}
\newtheorem{remark}[theorem]{Remark}
\newtheorem{definition}[theorem]{Definition}

\def\R{\mathbb{R}}

\def\P{\mathbb{P}}
\def\E{\mathbb{E}}
\def\T{\mathbb{T}}

\def\cC{\mathcal{C}}

\def\cC{\mathcal{C}}

\def\cF{\mathcal{F}}

\def\cL{\mathcal{L}}

\begin{document}

\title{Renormalization destroys a finite time bifurcation in the $\Phi^4_2$ equation}

\author{
  Alexandra Blessing Neam\c{t}u \thanks{University of Konstanz, Department of Mathematics and Statistics. Email: alexandra.blessing@uni-konstanz.de}
  \and
  Nicolas Perkowski \thanks{Institut f\"ur Mathematik, Freie Universit\"at Berlin and Max Planck Institute for Mathematics in the Sciences, Leipzig. Email: perkowski@math.fu-berlin.de}
  \and
  Chara Zhu \thanks{Institut f\"ur Mathematik, Freie Universit\"at Berlin. Email: xiaohan.zhu@fu-berlin.de }
}

\maketitle

\begin{abstract}
We study the singular $\Phi^4_2$ equation at a pitchfork bifurcation of the underlying deterministic dynamics.~To this aim, we linearize the SPDE along its stationary solution and show that the support of its finite-time Lyapunov exponents (FTLEs) is the real line, regardless of the bifurcation parameter and in sharp contrast to the non-singular $\Phi^4_1$ equation.~The proof relies on a support theorem for the stationary solution and its renormalized square. 
\end{abstract}

\section{Introduction and statement of the main result}

In this work we investigate finite-time Lyapunov exponents (FTLEs)\footnote{See Definition~\ref{def:ftle} for the definition of FTLEs in our setting.} for the $\Phi^4_2$ equation on the two-dimensional torus $\T^2$ 
given by 
\begin{align}\label{eq:phi} 
\begin{cases}\partial_t \Phi = \Delta \Phi -  :\Phi^3 : + \alpha \Phi + \xi, \\
\Phi(0)=\Phi_0,
\end{cases}
\end{align}
where $\xi$ denotes space-time white noise and the initial condition $\Phi_0$ belongs to a space of distributions. The term $\Phi^3 $ is subject to renormalization, since otherwise \eqref{eq:phi} is not well-posed in two dimensions. More precisely, convolving the noise with a smooth mollifier $\rho_\delta$, we obtain a family of equations 
\begin{align}\label{eq:phi} 
\begin{cases}\partial_t \Phi_\delta = \Delta \Phi_\delta - (\Phi^3_\delta-3 C_\delta \Phi_\delta) + \alpha \Phi_\delta   + \xi_\delta, \\
\Phi_\delta(0)=\Phi_0\star \rho_\delta,
\end{cases}
\end{align}
where the constants $C_\delta\to \infty$ as $\delta\to 0$.
Our goal is to analyze the FTLEs for the limit $\Phi=\lim\limits_{\delta\to 0} \Phi_\delta$ viewed as a stochastic flow taking values in a H\"older-Besov space of negative regularity.

More precisely, we aim to detect a phase transition / change of stability depending on the values of the parameter $\alpha$ for \eqref{eq:phi}.~These provide local expansion or contraction rates for the solutions of an SPDE starting from different initial data, indicating that nearby solutions are close to each other or tend to separate.~The motivation to quantity transitions for  SPDEs using FTLEs originates from the work of Crauel-Flandoli~\cite{CrFl:98} who show that the trajectories of the SDE
\begin{align}\label{sde} \text{d} x = (\alpha x -x^3)~\text{d}t +\sigma \text{d}W 
\end{align}
synchronize to a unique random fixed point, for any choice of $\alpha\in\R$ and noise intensity $\sigma>0$, inferring that additive noise destroys the deterministic pitchfork bifurcation occurring in $\alpha = 0$. More precisely, they show that the top asymptotic Lyapunov exponent is always negative.~However,
Callaway et al.~\cite{C} challenged this perspective by measuring local stability for trajectories of
the SDE~\eqref{sde} on finite time scales, using finite-time Lyapunov exponents. They proved that, whereas all FTLEs are negative for $ \alpha < 0,$ there is always a positive
probability to observe positive FTLEs for $\alpha>0$.~This probability was quantified in \cite{BBBE:25} using large deviations principles based on a spectral analysis of the Feynman-Kac semigroup associated to the linearized and projective dynamics.
For SPDEs, a similar transition in the probability distribution of the FTLEs has been detected for the $\Phi^4_1$ equation in~\cite{BEN23}.~Moreover, for $\Phi^4_1$, the stochastic Swift-Hohenberg equation, surface growth models and Burgers equation, such aspects were investigated in the small noise regime in \cite{BN:23,BB:25,Burgers} relying on the approximation of the corresponding SPDE with an SDE called amplitude equation. This approximation is valid on a large time horizon and allows one to transfer information about the FTLEs of the SDE to the SPDE.

While \cite{BEN23} shows that for $\Phi^4_1$ there is a change of sign in the FTLEs detecting a transition as $\alpha$ crosses the eigenvalues of the Laplace operator, we prove that this is not the case for the $\Phi^4_2$ equation:
\begin{theorem*}[Main result, see Theorem~\ref{thm:main}]
  The support of the finite time Lyapunov exponent of \eqref{eq:phi} is the real line, that is, for any value of $\alpha \in \R$ and $T>0$:
  \[
  \tmop{supp} (\mathbb{P}_{\lambda_T}) =\mathbb{R}.
  \]
\end{theorem*}
~This is an effect of the renormalization, which ``acts on the bifurcation parameter'': In~\eqref{eq:phi} the renormalization constant $C_\delta$ and the bifurcation parameter $\alpha$ play the same role. While there is a canonical choice for $C_\delta$ based on Wick renormalization, this is not sufficient to distinguish different values of $\alpha$. The proof relies on a support theorem for the stationary solution of \eqref{eq:phi} which is distributed according to the invariant measure of $\Phi^4_2$, and the argument is sufficiently robust to allow extensions to other equations with renormalization in the bifurcation parameter, for example pitchfork bifurcations in subcritical $\Phi^4$ models.~For $\Phi^4_1$, the support theorem approach recovers at least the first transition observed in \cite{BEN23}. 

 In future work, we plan to study the small noise regime and to approximate the dynamics of $\Phi^4_2$ close to a pitchfork bifurcation with an amplitude equation in the spirit of \cite{BN:23}.~Here we expect to observe a qualitative bifurcation since the renormalization affects only the stable modes and not the ones which are subject to a change of stability. We also plan to study the persistence of bifurcations under a renormalization that does not act directly on the bifurcation parameter.

\paragraph{Literature}

An alternative but related problem is to fix the bifurcation parameter $\alpha$
and introduce a coefficient $\varepsilon$ in front of the noise and determine for $\varepsilon\ll 1$ the extent to which the random motion resembles
that of the deterministic system, i.e.\ $\varepsilon=0$ when there are two stable solutions.~This situation has been investigated in~\cite{BeGe:13} for the $\Phi^4_1$ equation,  in \cite{BW:17} for the $\Phi^4_2$ equation, and in \cite{Klose2024} for the $\Phi^4_3$ equation.~For $\varepsilon>0$ the stable solutions become metastable and large deviations estimates provide a way of estimating the timescale of these transitions, giving another perspective on the persistence of
the bifurcation in the presence of noise.

There has been a strong interest in the dynamics of singular SPDEs. Here we mention \cite{GT20} who establish synchronization by noise for $\Phi^4_d$, where $d \in \{2,3\}$, using coming down from infinity estimates~\cite{Mourrat2017Global, Mourrat2017Dynamic} together with order-preservation similar to the 1D case considered in \cite{FGS}. Another synchronization type result based on large deviations was derived in \cite{Tsatsoulis2018} in the small noise regime, where solutions starting in a neighbourhood of $-1$ and $+1$ converge exponentially fast to each other.~Furthermore, concentration estimates around stable equilibria have been investigated for the $\Phi^4_2$ equation with time-dependent coefficients in \cite{BerglundNader} and similarly for $\Phi^4_3$ in \cite{Dimitri}. For ergodicity for the $\Phi^4_2$ equation on unbounded domains we refer to \cite{Bauerschmidt2025}, and for $\Phi^4_3$ to \cite{EP}.

Synchronization (in this context also called one force, one solution principle) for the KPZ equation on the torus was established in \cite{Tommaso} relying on order-preservation and Birkhoff's contraction principle, extending classical results \cite{Sinai1991,E2000} for the equation with spatially smooth noise to the singular white noise driven setting. The extension to the real line is given in~\cite{Dunlap} and \cite{KPZ} for the KPZ equation with spatially smooth or white noise, respectively.


\paragraph{Plan of the paper} This work is structured as follows. Section \ref{sec:p} collects basic concepts from paracontrolled calculus and white noise stochastic semi-flows. Section \ref{sec:main} contains our main result, Theorem \ref{thm:main}.~To this aim we establish a support theorem for the stationary solution of the $\Phi^4_2$ equation and its renormalized square.~Based on this, we derive in Corollary \ref{cor:const} that every constant function belongs to the support of the renormalized square and leads to the main result, Theorem \ref{thm:main}. This states that  
\[\text{supp}(\P_{\lambda_T})=\R, \]
where the random variable $\lambda_T$ is the finite-time Lyapunov exponent on a fixed time horizon $T>0$, as introduced in Definition \ref{def:ftle}. As already stated, this means that we cannot observe a change of sign in the FTLEs depending on the parameter $\alpha$.

For the convenience of the reader we provide in Appendix \ref{a1} the proof of the global well-posedness of the $\Phi^4_2$ equation from which we infer the continuity of certain maps that appear in the proof of the support theorem which relies on an auxiliary result stated in Appendix \ref{a2}.   

\paragraph{Acknowledgements}

The authors acknowledge funding by the Deutsche Forschungsgemeinschaft (DFG, German Research Foundation) - CRC/TRR 388 "Rough Analysis, Stochastic Dynamics and Related Fields" - Project ID 516748464. This material is based upon work supported by the National Science Foundation under Grant No. DMS-2424139, while AB and NP were in residence at the Simons Laufer Mathematical Sciences Institute in Berkeley, California, during the Fall 2025 semester.

\section{Preliminaries}\label{sec:p}

\subsection{Functional analytic tools}

We work with the following H\"older-Besov spaces of regularity $\beta\in\R$, see \cite[Section 2.7]{BCD:11}
      \begin{align}\label{besov:norm}
        \|f\|_{\mathcal{C}^\beta} := \| (2^{\beta k} \|\Delta_k f\|_{L^\infty})_{k\geq -1} \|_{l^\infty},
    \end{align}
    where $\Delta_k$ is the $k$-th Littlewoord-Paley block. For a fixed time horizon $T>0$ we write $\|f\|_\beta=\|f\|_{C_T\mathcal{C}^\beta}=\sup\limits_{t\in[0,T]}\|f(t)\|_\beta $ respectively $\|f\|_0=\|f\|_{C_TL^\infty}$. 
   
    We use the following useful estimates. For $\beta_1,\beta_2$ with $\beta_1+\beta_2>0$, the map $(u,v)\mapsto uv$ extends to a bounded linear form from $\mathcal{C}^{\beta_1\wedge \beta_2}
$ and
\begin{align}\label{pp} \| uv\|_{\cC^{\beta_1\wedge \beta_2}} \leq \|u\|_{\cC^{\beta_1}} \|v\|_{\cC^{\beta_2}}.  \end{align}

Moreover, the semigroup $(P_t)_{t\geq 0}$ generated by $\mathcal{L}:=\partial_t - \Delta$ satisfies for every $\beta\in\R$ and $\delta\geq 0$
\begin{align}\label{semigroup}
    \|P_t f\|_{\cC^{\beta+\delta}} \lesssim t^{-\delta/2} \|f\|_{\cC^\beta},~~t>0.
\end{align}


We write
$\Delta_{\leq j}=\sum\limits_{i\leq j }\Delta_i $ and similarly $\Delta_{\geq j}, \Delta_{>j}, \Delta_{<j}$. For two tempered distributions $u,v$ we define the {\em paraproduct }
\[ u \varolessthan v =\sum\limits_{j\geq -1} \Delta_{\leq j-2} u \Delta_j v \]
and the {\em resonant product}
\[ u \varodot v =\sum\limits_{i,j : |i-j|\leq 1} \Delta_i u \Delta_j v. \]
Their sum is denoted by 
\[  u \preccurlyeq v = u \varolessthan v + u\varodot v.\]
In the following, the notation $f\lesssim_a g $ means that there exists a constant $C=C(a)$ such that $f\leq C(a) g$.

\subsection{White noise stochastic semi-flows}
We fix $\varepsilon>0$ small. 
\begin{definition}
    Let $(\Omega,\cF,\P)$ be a probability space and $(\cF^t_s)_{t\geq s \geq 0}$ a family of sub $\sigma$-algebras of $\cF$. We call a family $\{\Phi(t;s;\cdot):t\geq s \geq 0\}$ of $\cF^t_s$-measurable maps
    \[ \Phi(t;s;\cdot)(\omega) : 
    \mathcal{C}^{-\varepsilon} \ni \Phi_0 \mapsto \Phi(t;s;\Phi_0)(\omega)\in \mathcal{C}^{-\varepsilon} \]
    a white noise stochastic semi-flow if the following properties hold.
    \begin{itemize}
        \item [1)] The $\sigma$-algebras $\cF^{t}_s$ and $\cF^{t'}_{s'}$ are independent for all $s\leq t \leq s'\leq t'$.
        \item [2)] The family $\{ \Phi(t;s;\cdot): t\geq s\geq 0 \}$ satisfies the flow property, meaning that 
        \[ \Phi(t+s;0;\Phi_0) =\Phi(t;s;\Phi(s;0;\Phi_0)),~\text{for every } t\geq s \geq 0, \Phi_0\in\mathcal{C}^{-\varepsilon}, \omega\in\Omega. \]
        \item [3)] For all $\Phi_0\in \mathcal{C}^{-\varepsilon}$ and $t\geq s \geq 0$, $\Phi(t;s;\Phi_0)$ and $\Phi(t-s;0;\Phi_0)$ have the same law. 
    \end{itemize}
\end{definition}
   
We now introduce a two-sided filtration associated to space-time white noise. This is denoted by $\mathcal{F}_s^t $ and is the augmentation of
 \[ \mathcal{\tilde{F}}^t_s = \sigma (\{ \xi(\phi) : \phi \in L^2(\mathbb{R}\times \mathbb{T}^2  ) ~\text{s.t. } \phi_{[0,s) \cup (t,\infty)\times \mathbb{T}^2 }=0  \}  ),~~ t\geq s\geq 0 \]
  which can be extended for $t\geq s>-\infty.$

\begin{theorem}\label{thm:a:alpha}
    There exists a stationary family of random variables $\{a_{\alpha}(t)\}_{t\geq 0}$ in $\mathcal{C}^{-\varepsilon}$ such that $a_\alpha(0)$ is independent of $\mathcal{F}^\infty_0$ and the law of $a_\alpha(0)$ is the invariant measure of~\eqref{eq:phi}.
\end{theorem}
\begin{proof}
This follows by~\cite[Corollary 1.2]{GT20}. We sketch the main ideas of the proof. Similar to~\cite[Corollary 1.2]{GT20}, one can show that the family $\{\Phi(t;s;\cdot)\}_{t\geq s\geq 0}$ generates a white noise stochastic semi-flow $\Phi$. 
We extend the time interval to $-\infty$ and set
\[  a_\alpha(t;s):=\Phi(t;s;0) \text{ for } t\geq s>-\infty. \]
Moreover, using~\cite[Theorem 1.1]{GT20}, one can show that $(a_\alpha(t;s))_{s\leq 0}$ is a Cauchy sequence in $L^p(\Omega, \mathcal{C}^{-\varepsilon})$ for $p>2/\gamma$ where $0<\gamma<\varepsilon$.
We set 
\[ a_\alpha(t):=\lim\limits_{s\searrow -\infty} a_\alpha(t;s).\]
 Due to the flow property of $\Phi$, we have the following equality in law
\[ \Phi(t;s;0) =\Phi(t-s;0;0)=\Phi(0;-(t-s);0). \]
Since $\Phi(t;s;0)\to a_\alpha(t)$ and $\Phi(0;-(t-s);0)\to a_\alpha(0)$, as $s\searrow -\infty$, in $L^p(\Omega,\mathcal{C}^{-\varepsilon})$, we obtain the stationarity of $a_\alpha(t)$. Furthermore,
in order to show that the law of $a_\alpha(0)$ is the invariant measure of~\eqref{eq:phi} we consider an arbitrary $F\in C_b(\mathcal{C}^{-\varepsilon})$ and obtain using again the flow property of $\Phi$ that
\begin{align*}
    \E (\E ( F(\Phi(t;\Phi_0) ))|_{\Phi_0=a_\alpha(0)}&=\lim\limits_{s\to -\infty} \E (\E F(\Phi(t;\Phi_0))|_{\Phi_0 =\Phi(0;s;0)} =\lim\limits_{s\to-\infty} \E (F(\Phi(t;s;0))\\
    & = \lim\limits_{s\to-\infty} \E F(\Phi(0;-(t-s);0)=\E F (a_\alpha(0)).
\end{align*}
By construction, $a_\alpha(0)$ is $\mathcal{F}^0_{-\infty}$-measurable and therefore independent of $\mathcal{F}_0^\infty$.
\end{proof}
Let $v $ be the solution of the linearization of~\eqref{eq:phi} around $a_\alpha(\cdot)$ given by
\begin{align}\label{linearization}
\begin{cases}    \partial_t v = \Delta v + \alpha v - 3 : a^2_{\alpha}(t) : v \\v (0) =
   v_0 \in L^2. 
   \end{cases}
\end{align} 

\begin{definition}\label{def:ftle}
    We call the following random variable $\lambda_T$ the finite-time Lyapunov exponent at time $T$:
\begin{align}\label{ftle} \lambda_T: = \frac{1}{T} \log \| v (T) \|_{L (L^2, L^2)}= \frac{1}{T}  \log \sup_{v_0
   \in L^2  : \| v_0 \|_{L^2 } = 1} \| v (T)
   \|_{L^2}.
   \end{align}
\end{definition}
\begin{remark}
    Note that while $\Phi\in C_T\cC^{-\varepsilon}$, the solution of the linearization belongs to $C_T\cC^{2 - \varepsilon}$.
\end{remark}
In the next section we analyze the support of $\lambda_T$ given the support of the renormalized square $:a_\alpha^2:$. 

\section{Main result}\label{sec:main}

Due to Theorem~\ref{thm:a:alpha}, we obtain a support theorem for $(a_{\alpha}
(t))_{t \in [0, T]}$. To state it, let us define the following maps.

\begin{definition}
The map $\mathcal{J}_{\alpha}$ is defined as
\[ \mathcal{J}_{\alpha} : H^1 (\mathbb{T}^2) \times L^2 ([0, T] \times
   \mathbb{T}^2) \times [0, \infty) \to C_T \cC^{-
   \varepsilon}, \quad \mathcal{J}_{\alpha} (\varphi_0, f,
   c) := \Phi^{\varphi_0, f, c}, \]
where $\Phi^{\varphi_0, f, c}$ is the solution to
\[ \partial_t \Phi^{\varphi_0, f, c} = \Delta \Phi^{\varphi_0, f, c} -
   (\Phi^{\varphi_0, f, c})^3 + 3 c \Phi^{\varphi_0, f, c} + \alpha
   \Phi^{\varphi_0, f, c} + f, \quad \Phi^{\varphi_0, f, c}_0 = \varphi_0 .
\]
\end{definition}

\begin{definition}\label{def:Psi}
For any $\varepsilon > 0$, the map $\Psi_{\alpha}$ is defined as
\[ \Psi_{\alpha}: \cC^{-\varepsilon} \times (C_T\cC^{-\varepsilon})^3 \to C_T\cC^{-\varepsilon}, \qquad  \Psi_{\alpha}(\varphi_0, (Z_1, Z_2, Z_3)) := R^{\varphi_0, \mathcal Z} + Z_1, \]
where $\mathcal Z = (Z_1, Z_2, Z_3)$ and the remainder $R^{\varphi_0, \mathcal Z}$ is the solution to
\begin{equation} \label{eq:R}
    \partial_t R^{\varphi_0, \mathcal Z} = \Delta R^{\varphi_0, \mathcal Z} - (R^{\varphi_0, \mathcal Z})^3 - 3(R^{\varphi_0, \mathcal Z})^2 Z_1 - 3R^{\varphi_0, \mathcal Z} Z_2 - Z_3 + \alpha R^{\varphi_0, \mathcal Z}, \quad R^{\varphi_0, \mathcal Z}(0) = \varphi_0.
\end{equation}   
\end{definition}

\begin{theorem}[Continuity of the Solution Map]\label{thm:continuity}
The map $\Psi_\alpha$ introduced above satisfies the following continuity properties.
\begin{enumerate}
    \item The remainder map $\mathcal R: (\varphi_0, (Z_1, Z_2, Z_3)) \mapsto R^{\varphi_0, \mathcal{Z}}$ is continuous from $\cC^{-\varepsilon} \times (C_T\cC^{-\varepsilon})^3$ into $C_T\cC^{2-\varepsilon}$.
    \item The full solution map $\Psi_\alpha: (\varphi_0, (Z_1, Z_2, Z_3)) \mapsto R^{\varphi_0, \mathcal Z} + Z_1$ is continuous from $\cC^{-\varepsilon} \times (C_T \cC^{-\varepsilon})^3$ into $C_T\cC^{-\varepsilon}$.
\end{enumerate}
\end{theorem}

\begin{proof}
    See Appendix \ref{a1}.
\end{proof}

The support theorems for $a_\alpha$ and $:a^2_\alpha:$ rely on the following result. 

\begin{lemma}\label{lem:supp}
Let $E$ and $F$ be metric spaces and let $\varphi:E\to F$ be continuous. Let $X$ be an $E$-valued random variable and denote by $\mathbb{P}_X$ its law, as a measure on the Borel sets of $E$. Denote $\mathbb{P}_{\varphi(X)} = \mathbb{P}_X \circ \varphi^{-1}$ the pushforward. Assume that $\mathbb{P}_X$ is a Radon probability measure (e.g., this holds if $E$ is Polish). Then
\[ \tmop{supp}(\mathbb{P}_{\varphi(X)}) = \overline{\varphi(\tmop{supp}(\mathbb{P}_X))}. \]
\end{lemma}

\begin{proof}
    See Appendix \ref{a2}.
\end{proof}

We now establish the support theorems for $a_\alpha$ and its renormalized square.

\begin{theorem}[Support theorem of $a_\alpha$]\label{thm:supp_a_alpha}
  The support of the law $\mathbb{P}_{(a_{\alpha} (t))_{t \in [0,
  T]}}$ is
  \[ \tmop{supp} (\mathbb{P}_{(a_{\alpha} (t))_{t \in [0, T]}}) = \overline{\{
     \mathcal{J}_{\alpha} (\varphi_0, f, c) : \varphi_0 \in H^1, f \in L^2
     ([0, T] \times \mathbb{T}^2), c \in [0, \infty) \}}, \]
  where the closure is taken in $C_T\cC^{- \varepsilon}$.
\end{theorem}

\begin{proof}
Let
\[
    S_0 := \mathcal{C}^{-\varepsilon}_0:=\overline{H^1}^{\,\cC^{-\varepsilon}} = \Big\{ \varphi \in \cC^{-\varepsilon} : \lim_{j \to \infty} 2^{-j\varepsilon} \|\Delta_j \varphi \|_{L^\infty} = 0\Big\},
\]
and
\[ A:=\{(h,h^2-c,h^3-3ch):h\in\mathcal H_0(T), c\in[0,\infty)\}\subset (C_T\cC^{-\varepsilon})^3, \]
where 
\[ \mathcal H_0(T) :=\{ h(\cdot)=\int_0^{\cdot} e^{(\cdot-r)(\alpha+\Delta)} f(r)\,\mathrm{d}r : f\in L^2([0,T]\times\mathbb T^2)\}.\]
By Theorem~\ref{thm:a:alpha}, the initial datum $a_\alpha (0)$ is independent of $\cF^\infty_0$ and distributed according to the invariant measure of the $\Phi^4_2$ equation denoted by $\mu_{\Phi^4_2}$. This measure has the density
\[ \rho(\varphi) = \frac{1}{Z} \exp \left( - \frac{1}{2} \int_{\mathbb{T}^2} (: \varphi^4 :
     - (\alpha + 1) \varphi^2) (x) \mathd x \right) \nu (\mathd \varphi) \]
with respect to the massive Gaussian free field $\nu \sim \mathcal{N} (0, (1 - \Delta)^{- 1})$. 
This construction follows the Euclidean strategy for quantum field theory, as detailed in \cite[Chapter~V]{Simon2015}. Since the density $\rho$ is strictly positive $\nu$-almost everywhere, the measures $\mu_{\Phi^4_2}$ and $\nu$ are mutually absolutely continuous, and thus
\[\tmop{supp}(\mu_{\Phi^4_2}) = \tmop{supp}(\nu).\] 

Now we show for completeness that $\tmop{supp}(\nu) = S_0$, where $S_0 = \mathcal{C}^{-\varepsilon}_0$ as defined previously. 
The sample paths of the Gaussian free field lie in $S_0$ almost surely, hence $\nu(S_0)=1$. 
Since $S_0$ is a separable Banach space, the restriction $\nu|_{S_0}$ is a Radon Gaussian measure on $S_0$. The Cameron-Martin space of $\nu$ is $H^1$, so by \cite[Theorem 3.6.1]{Bogachev1998}, the support of $\nu$ within the topology of $S_0$ is $\tmop{supp}_{S_0}(\nu) = \overline{H^1}^{S_0}$, which by construction is $S_0$ itself. 
Since $S_0$ is a closed subset of $\cC^{-\varepsilon}$ and $\nu(S_0)=1$, it follows that the support of $\nu$ in $\cC^{-\varepsilon}$ is also $S_0$.
By \cite[Theorem 6.3]{Tsatsoulis2018} (adapted to $[0,T]$), the law of the Wick powers satisfies
\[ \tmop{supp}(\mathbb{P}_{(Z,:Z^2:,:Z^3:)})=S_1:=\overline{A} \subset (C_T\cC^{-\varepsilon})^3, \]
with closure taken in $(C_T\cC^{-\varepsilon})^3$. Since $H^1$ and $\mathcal H_0(T)$ are separable, both $S_0$ and $S_1$ are closed separable subsets of their ambient spaces, hence Polish. 
By the independence of $a_\alpha(0)$ and of $\cF^\infty_0$, the joint law $\mathbb{P}_{(a_\alpha(0),Z,:Z^2:,:Z^3:)}$ has support
\[\tmop{supp}(\mathbb{P}_{(a_\alpha(0),Z,:Z^2:,:Z^3:)}) = S_0 \times S_1\]
and is a Radon probability measure on the Polish space $ S_0\times S_1 $.

An element in $A$ is a triplet of the form $(h,h^2-c,h^3-3ch)$ where $h \in \mathcal H_0(T)$ and $c\in[0,\infty)$. By definition, each $h \in \mathcal H_0(T)$ arises as the unique mild solution to the linear PDE
\[\partial_th=(\alpha + \Delta)h + f, \qquad h(0)=0,\]
for some $f \in L^2([0, T] \times \mathbb{T}^2)$. This establishes a correspondence between $(h,h^2-c,h^3-3ch)$ and $(f,c)$. Therefore, on the dense set $H^1 \times A$, the map $\Psi_\alpha$ coincides with $\mathcal J_\alpha$ with the corresponding $f$, i.e.
\[ \Psi_\alpha: H^1\times A \to C_T\cC^{-\varepsilon},\qquad \Psi_\alpha(\varphi_0,(h,h^2-c,h^3-3ch))=\mathcal J_\alpha(\varphi_0,f(h),c). \]

By Theorem~\ref{thm:continuity}, the map $\Psi_\alpha$ is continuous on $S_0 \times S_1$ with respect to the product topology induced by $H^1$ and $(C_T\cC^{-\varepsilon})^3$.  Now apply Lemma~\ref{lem:supp} to the Radon measure $\mathbb{P}_{(a_\alpha(0),Z,:Z^2:,:Z^3:)}$ on the Polish space $E:=S_0 \times S_1$. We obtain
\[ \tmop{supp}(\mathbb{P}_{\Psi(a_\alpha(0),Z,:Z^2:,:Z^3:)}) = \overline{\Psi_\alpha(\tmop{supp}(\mathbb{P}_{(a_\alpha(0),Z,:Z^2:,:Z^3:)})} = \overline{\Psi_\alpha(S_0\times S_1)}, \]
with closure in \(C_T\cC^{-\varepsilon}\). 
By the definition of $\Psi_\alpha$ on $H^1\times A$, we have
\[  \Psi_\alpha(H^1\times A) = \{ \mathcal{J}_{\alpha} (\varphi_0, f, c) : \varphi_0 \in H^1, f \in L^2([0, T] \times \mathbb{T}^2), c \in [0, \infty) \},  \]
and therefore their closures are identical
\[  \overline{\Psi_\alpha(H^1\times A)} = \overline{\{ \mathcal{J}_{\alpha} (\varphi_0, f, c) : \varphi_0 \in H^1, f \in L^2([0, T] \times \mathbb{T}^2), c \in [0, \infty) \}}.  \]
Since $H^1\times A$ is a dense subset of $S_0 \times S_1$ and $\Psi_\alpha$ is continuous,
\[ \overline{\Psi_\alpha(H^1\times A)}=\overline{\Psi_\alpha(S_0\times S_1)}. \]
Combining these equalities, we obtain
\[ \tmop{supp} (\mathbb{P}_{(a_{\alpha} (t))_{t \in [0, T]}}) = \overline{\{
     \mathcal{J}_{\alpha} (\varphi_0, f, c) : \varphi_0 \in H^1, f \in L^2
     ([0, T] \times \mathbb{T}^2), c \in [0, \infty) \}}, \]
where the closure is taken in $C_T\cC^{- \varepsilon}$.
\end{proof}

\begin{corollary}[Support theorem of ${:}a_\alpha^2{:}$]
  The support of the law $\mathbb{P}_{(:a_\alpha^2:(t))_{t\in[0,T]}}$ is
  \[ \tmop{supp}(\mathbb{P}_{(:a_\alpha^2:(t))_{t\in[0,T]}}) = \overline{\{\mathcal{J}_\alpha(\varphi_0,f,c)^2 - c : \varphi_0 \in H^1, f \in L^2([0,T]\times \mathbb{T}^2), c \in [0,\infty)\}}, \]
  where the closure is taken in $C_T\cC^{-\varepsilon}$.
\end{corollary}

\begin{proof}

Consider the decomposition
\[
:a_\alpha^2:\;=\; (a_\alpha-Z)^2 + 2 (a_\alpha-Z)\,Z + Z^{:2:} = F(a_\alpha - Z, Z, Z^{:2:}),
\]
where $F(u,v,w) := u^2 + 2uv + w$.
By \eqref{pp}, $F$ defines a continuous map
\[
\cC^{2-\varepsilon} \times \cC^{-\varepsilon} \times \cC^{-\varepsilon} \to \cC^{-\varepsilon},
\]
and hence a continuous map on
\[
C_T\cC^{2-\varepsilon} \times (C_T\cC^{-\varepsilon})^2 \to C_T\cC^{-\varepsilon}.
\]
Applying Lemma~\ref{lem:supp} to the continuous map $F$ and noting that
\[ F(\mathcal{J}_\alpha(\varphi_0,f,c) - h, h, h^2-c) = \mathcal{J}_\alpha(\varphi_0,f,c)^2 - c, \]
we obtain the desired result.
\end{proof}

\begin{corollary} \label{cor:const}
  For any $\kappa \in \mathbb{R}$ the constant function with $\psi (t, x)
  \equiv \kappa$ for all $(t, x) \in [0, T] \times \mathbb{T}^2$ is in the
  support of $\mathbb{P}_{(: a_{\alpha}^2 (t) :)_{t \in [0, T]}}$.
\end{corollary}

\begin{proof}
  Let $\kappa \in \mathbb R$. Note that spatial constant functions are in $H^1$, and constant space-time functions are in $L^2([0,T]\times \mathbb T^2)$. Therefore, it suffices to find constant functions $\varphi_0, f$ and $c \geqslant 0$
  such that the equation
  \[ \partial_t \Phi = \Delta \Phi - \Phi^3 + (3 c + \alpha) \Phi + f, \qquad
     \Phi (0) = \varphi_0, \]
  is constantly equal to $\varphi_0$ with $\varphi_0^2 - c = \kappa$.
  Thus, we require
  \[ - \varphi_0^3 + (3 c + \alpha) \varphi_0 + f = 0, \qquad \varphi_0^2 - c
     = \kappa . \]
  We resolve the second equation to $c = \varphi_0^2 - \kappa$, which for
  $\kappa \geqslant 0$ requires $| \varphi_0 | \geqslant \sqrt{\kappa}$ to
  guarantee $c \geqslant 0$. We plug this identity in the first equation and
  obtain
  \[ f = - 2 \varphi_0^3 - (\alpha - 3 \kappa) \varphi_0 . \]
  Thus, for $\kappa < 0$ we can take $\varphi_0 = 0$, $f =
  0$, and $c = - \kappa$. For $\kappa \geqslant 0$ we can take $\varphi_0 = \sqrt{\kappa}$, $c = 0$
  and $f = - 2 \kappa^{3 / 2} - \alpha \kappa^{1 / 2} + 3 \kappa^{3 / 2} =
  \kappa^{3 / 2} - \alpha \kappa^{1 / 2}$. In both cases, we obtain $\mathcal{J}_{\alpha} (0, 0, c)^2 - c \equiv \kappa$.
\end{proof}

\begin{lemma} \label{lem:sol}
  For $\kappa \in \mathbb{R}$ and $v_0 \in L^2 $, the solution
  $v$ to the linear equation
  \[ \partial_t v = \Delta v + (\alpha - 3 \kappa) v, \qquad v (0) = v_0 \]
  is given by $v (t) = e^{t (\alpha - 3 \kappa)} P_t v_0$, where $(P_t)_{t \geqslant 0}$ is the semigroup generated by $\cL$. In
  particular,
  \[ \sup_{\| v_0 \|_{L^2 } = 1} \| v (t) \|_{L^2
     } = e^{t (\alpha - 3 \kappa)} . \]
\end{lemma}

\begin{proof}
  The identity $v(t)=e^{t (\alpha - 3 \kappa)} P_t v_0$ follows by differentiating the right hand side. Since $P_t$ is a contraction in $L^2$ for each $t \geqslant 0$, which follows by Jensen's
  inequality or Parseval's identity for the Fourier transform, we obtain the
  upper bound
  \[ \| v (t) \|_{L^2} = \| e^{t (\alpha - 3 \kappa)} P_t v_0
     \|_{L^2 } \leqslant e^{t (\alpha - 3 \kappa)} \| v_0
     \|_{L^2 } . \]
  For the lower bound we take $v_0 \equiv 1$ for which $P_t v_0 = v_0$ and
  thus $v (t) = e^{t (\alpha - 3 \kappa)} v_0$.
\end{proof}

\begin{theorem}\label{thm:main}
  The support of the finite time Lyapunov exponent of \eqref{eq:phi} is the real line, that is, for any value of $\alpha \in \R$ and $T>0$:
  \[
  \tmop{supp} (\mathbb{P}_{\lambda_T}) =\mathbb{R}.
  \]
\end{theorem}

\begin{proof}
   Recall that the FTLE $\lambda_T$ is defined as
\[ \lambda_T = \frac{1}{T} \log \| v (T) \|_{L (L^2 , L^2
   )} \]
with
\[ \| v (T) \|_{L (L^2 , L^2 )} := \sup_{\| v_0 \|_{L^2 } = 1} \| v (T)
   \|_{L^2 } . \]
We consider the map
\[\Lambda:C_T\cC^{-\varepsilon} \to \mathbb{R}, \qquad \Lambda(\psi):=\frac{1}{T}\log\| v^\psi (T) \|_{L (L^2, L^2)},\]
where $v:=v^\psi$ solves $\partial_t v = \Delta v +(\alpha - 3\psi)v$ with $v(0)=v_0$. 
Applying Lemma~\ref{lem:supp} to this continuous map, we have
\[ \tmop{supp} (\mathbb{P}_{\lambda_T}) = \overline{\Lambda(\tmop{supp}(\mathbb{P}_{: a_{\alpha}^2 :}))}\]
Let $\lambda \in \mathbb{R}$ and let $\kappa = \frac{\alpha - \lambda}{3}$. 
By Corollary~\ref{cor:const}, the constant function with $\psi (t, x) \equiv \kappa$ for all $(t, x) \in [0, T] \times \mathbb{T}^2$ is in the support of $\mathbb{P}_{:a_{\alpha}^2:}$.
By Lemma~\ref{lem:sol}, 
\[\Lambda(\kappa)=\frac{1}{T}\log\| v^\kappa (T) \|_{L (L^2, L^2)} =\frac{1}{T}\log(e^{T(\alpha-3\kappa)}) =\lambda.\]
Therefore, $\lambda \in \tmop{supp} (\mathbb{P}_{\lambda_T})$. As $\lambda \in \mathbb{R}$ is arbitrary, we have $\mathbb{R} \subset \tmop{supp} (\mathbb{P}_{\lambda_T})$. Since $\lambda_T$ is a real-valued random variable, we have $\tmop{supp} (\mathbb{P}_{\lambda_T}) =\mathbb{R}$.
\end{proof}

\appendix

\section{Appendix}
\subsection{Proof of Theorem~\ref{thm:continuity}}\label{a1}

The following global well-posedness result for the nonlinear residual equation is well known, see for example \cite{DaPrato2003, Mourrat2017Dynamic, Gubinelli2019Global}. For completeness we include the arguments, based on the approach developed in \cite{Gubinelli2019Global, Jagannath2023}.
 The reminder $R=R^{\varphi_0,\mathcal Z}$ solves \eqref{eq:R}. The linear term $\alpha R$ can be handled with only minor modifications. For simplicity, we omit it in the statements and proofs of Theorems~\ref{thm:local} and~\ref{thm:global}, and focus on the following equation
\begin{align} \label{eqn_u}
\left\{
    \begin{aligned}
    &\mathcal{L}  R= - R^3 -3R^2Z_1-3RZ_2-Z_3 \text{,}\quad \text{on} \; \mathbb{R}_+ \times \mathbb{T}^2 \\
    &R(0) = \varphi_0
    \end{aligned}
\right..
\end{align}

\begin{theorem}[Local well-posedness]\label{thm:local}
For $\varphi_0 \in \mathcal{C}^{2-\varepsilon}$, there exists $T^* > 0 $ such that for all $T<T^*$ the equation~\eqref{eqn_u} has a unique solution $u \in C_T \mathcal{C}^{2-\varepsilon}$, and this solution depends continuously on $\mathcal Z=(Z_1,Z_2,Z_3) \in \cap_{\varepsilon' > 0}(C_T\mathcal{C}^{-\varepsilon'})^3$ and $\varphi_0 \in \mathcal{C}^{2-\varepsilon}$. Moreover, if $T^* < \infty$, then $\lim_{t\to T^*}\|R(t)\|_{\mathcal{C}^{2-\varepsilon}} = \infty$. 
\end{theorem}
\begin{proof}
    Set up a Picard iteration
    \[
    \mathcal I(R)(t) = P_t \varphi_0 + \int^t_0 P_{t-s} (- R^3 -3R^2Z_1-3RZ_2-Z_3)(s)\text{d}s.
    \]
    For any $R_1$, $R_2 \in C_T\mathcal{C}^{2-\varepsilon}$, we have using \eqref{semigroup} for $\beta=-\varepsilon/2$ and $\delta=2-\varepsilon/2$
    \begin{align*}
    \|\mathcal I(R_1)-\mathcal I(R_2)\|_{2-\varepsilon} &\lesssim T^{\frac{\varepsilon}{4}}\|-( R_1^3-R_2^3) -3(R_1^2-R_2^2)Z_1-3(R_1-R_2)Z_2\|_{-\frac{\varepsilon}{2}}\\
    &\lesssim_{\mathcal Z} T^{\frac{\varepsilon}{4}}(1+\|R_1\|^2_{2-\varepsilon}+\|R_2\|^2_{2-\varepsilon})\|R_1-R_2\|_{2-\varepsilon}.
    \end{align*}
    Let $M = 2 \|\varphi_0\|_{\mathcal{C}^{2-\varepsilon}}$. For $T>0$ sufficiently small, depending only on $M$ and $\mathcal Z$, the map $\mathcal I$ is a contraction on $B_{C_T\mathcal{C}^{2-\varepsilon}}(0,M)$. By the Banach fixed point theorem, we can find a unique solution on the interval $[0,T]$, where $T$ depends on $M$ and $\mathcal Z$. We can iterate the construction and get a maximal existence time $T^* \in (0,\infty]$. The continuous dependence on the data follows from similar fixed point estimates. For $\mathcal{Z},\mathcal{\tilde{Z}}$, $\varphi_0,\tilde{\varphi}_0$ we consider the fixed points $R$ and $\tilde{R}$ of the maps $I_{\mathcal{Z},\varphi_0}$ respectively $I_{\mathcal{\tilde{Z}},\tilde{\varphi}_0}$ and write $R-\tilde{R}= I_{\mathcal{Z},\varphi_0}(R) -I_{\mathcal{\tilde{Z}}, \tilde{\varphi}_0}(\tilde{R})$ to bound $R-\tilde{R}$ in terms of $\mathcal{Z}-\tilde{\mathcal{Z}}$ respectively $\varphi_0 -\tilde{\varphi}_0$. 
\end{proof}

\begin{theorem}[Global well-posedness]\label{thm:global}
Let $R$ and $T^*$ be as in Theorem~\ref{thm:local} and let $T < T^* \wedge 1$. There exists a constant $C>0$ that depends only on $\varphi_0$ and $\mathcal Z=(Z_1,Z_2,Z_3)$ but not on $T$, such that
\[\|R\|_{2-\varepsilon} \leq C.\]
In particular, $T^* = \infty$.
\end{theorem}
For the proof of Theorem~\ref{thm:global}, we need uniform bounds for the solution $R$ on time intervals of length $T$. The key idea is to decompose the solution into a more regular part and a singular part using a paraproduct decomposition, which allows us to separate the most irregular terms and and exploit their different regularities.\\
\textbf{Paraproduct Decomposition.}
\begin{align*}
    \mathcal{L} R_1 &=U_1 (R_1,R_2) \\
    \mathcal{L} R_2 &= - R_2^3 + U_2 (R_1,R_2) ,
\end{align*}
with initial conditions $R_1(0)= R(0) = \phi_0$ and $R_2(0) = 0$, where
\begin{align*}
    U_1 (R_1,R_2) &= -3R^2 \olessthan \Delta_{>2n}Z_1 - 3R \olessthan \Delta_{>n}Z_2 - Z_3\\
    U_2 (R_1,R_2) &=  -3R^2 \olessthan \Delta_{\leq 2n}Z_1 - 3R \olessthan \Delta_{\leq n}Z_2 - 3R^2 \succcurlyeq Z_1 - 3R \succcurlyeq Z_2 - (R^3 - R_2^3),
\end{align*}
for $R := R_1 + R_2$. In this way, $U_1$ collects singular terms, whereas $U_2$ is more regular. Note that since $R_1 + R_2$ solves the equation~\eqref{eqn_u}, by the uniqueness in Theorem~\ref{thm:local} we have $R_1 + R_2 = R$.
\begin{lemma} \label{lem:U}
    For any $\delta \geq 0$, $n\geq 1$, $\varepsilon'>0$, we have
    \begin{align*}
    \|U_1\|_{-\varepsilon-\delta} &\lesssim_\mathcal Z (1+2^{-n\delta}\|R_1 + R_2\|_0)^2. \\
     \|U_2\|_{-\varepsilon} &\lesssim_\mathcal Z \sum\limits^3_{i=1}\|R_1\|_0^i\cdot\|R_2\|_0^{3-i} + \|R_1 + R_2\|_{\varepsilon'}(1+\|R_1 + R_2\|_0).
    \end{align*}
    Furthermore, for any $\kappa > 0$, $n\geq 1$, $\varepsilon'>0$, 
    \begin{align*}
    \|U_2\|_0 &\lesssim_\mathcal Z \sum\limits^3_{i=1}\|R_1\|_0^i\cdot\|R_2\|_0^{3-i} + \|R_1 + R_2\|_{\varepsilon'}(1+\|R_1 + R_2\|_0)\\
    &\quad + 2^{n\kappa}\|R_1 + R_2\|_0 + 2^{2n\kappa}\|R_1 + R_2\|_0^2.
    \end{align*}
\end{lemma}
\begin{proof}
    Notice that for any $\beta \in \mathbb{R}$, $f \in \mathcal{C}^\beta$, and $\delta > 0$, we have
    \[
    \|\Delta_{>n}f\|_{\mathcal{C}^{\beta-\delta}} \lesssim 2^{-n\delta}\|f\|_{\mathcal{C}^\beta} \quad \text{and} \quad \|\Delta_{\leq n}f\|_{\mathcal{C}^{\beta+\delta}} \lesssim 2^{n\delta}\|f\|_{\mathcal{C}^\beta}.
    \]
    For $U_1 = -3R^2 \olessthan \Delta_{>2n}Z_1 - 3R \olessthan \Delta_{>n}Z_2 - Z_3$, using paraproduct estimates, we obtain
    \begin{align*}
    \|U_1\|_{-\varepsilon-\delta} &\lesssim \|R_1+R_2\|^2_0 \cdot \|\Delta_{>2n}Z_1\|_{-\varepsilon-\delta} + \|R_1+R_2\|_0 \cdot \|\Delta_{>n}Z_2\|_{-\varepsilon-\delta} + \|Z_3\|_{-\varepsilon-\delta}\\
    &\lesssim_\mathcal Z (1+2^{-n\delta}\|R_1 + R_2\|_0)^2.
    \end{align*}
    Now for the bound of $U_2$, since $Z_1$, $Z_2 \in C_T\mathcal{C}^{-\varepsilon'}$, we have 
    \[
    \|R^2 \succcurlyeq Z_1\|_0 + \|R \succcurlyeq Z_2\|_0 \lesssim_\mathcal Z \|R^2\|_{\varepsilon'} + \|R\|_{\varepsilon'} \lesssim \|R\|_{\varepsilon'} (1 + \|R\|_0).
    \]
    Then, as 
    \[\|R^2 \olessthan \Delta_{\leq 2n}Z_1\|_{-\varepsilon} + \|R \olessthan \Delta_{\leq n}Z_2\|_{-\varepsilon}\lesssim_\mathcal Z\|R^2\|_0+\|R\|_0,
    \]
    we have the bound $$\|U_2\|_{-\varepsilon} \lesssim_\mathcal Z \sum\limits^3_{i=1}\|R_1\|_0^i\cdot\|R_2\|_0^{3-i} + \|R_1 + R_2\|_{\varepsilon'}(1+\|R_1 + R_2\|_0). $$
    For $\kappa' \in (0,\varepsilon\wedge\kappa)$,
    \[
    \|R^2 \olessthan \Delta_{\leq 2n}Z_1\|_{-\kappa'+\kappa} + \|R \olessthan \Delta_{\leq n}Z_2\|_{-\kappa'+\kappa}\lesssim_\mathcal Z2^{2n\kappa}\|R^2\|_0+2^{n\kappa}\|R\|_0.
    \]
    Therefore, 
    \begin{align*}
    \|U_2\|_0 &\lesssim_\mathcal Z \sum\limits^3_{i=1}\|R_1\|_0^i\cdot\|R_2\|_0^{3-i} + \|R_1 + R_2\|_{\varepsilon'}(1+\|R_1 + R_2\|_0)\\
    &\quad + 2^{n\kappa}\|R_1 + R_2\|_0 + 2^{2n\kappa}\|R_1 + R_2\|_0^2.
    \end{align*}
\end{proof}

\begin{corollary} \label{cor:u}
    If $n \geq 0$ is such that $2^{-n(2-2\varepsilon)}(\|R_1+R_2\|_0\vee1)\simeq 1$, then for all $\beta \in[0,2-\varepsilon]$ we have
\begin{align}
    &\|R_1\|_\beta \lesssim_{\mathcal Z,\varphi_0} 1+\|R_2\|_0^{\frac{2\beta}{2-\varepsilon}} \label{C4_1}\\
    &\|R_2\|_0 \lesssim_{\mathcal Z,\varphi_0} 1+\|R_2\|_{\varepsilon'}^{\frac{1}{2}} \label{C4_2}\\
    &\|R_2\|_{2-\varepsilon} \lesssim_{\mathcal Z,\varphi_0} 1+\|R_2\|_0^3. \label{C4_3}
\end{align}
\end{corollary}
\begin{proof}
    Recall that $R_1(0) = \varphi_0 \in \mathcal{C}^{2-\varepsilon}$. By Schauder estimate and Lemma~\ref{lem:U}, we obtain for any $\delta \in [0,2-\varepsilon)$
    \[
    \|R_1\|_{2-\varepsilon-\delta} \lesssim_{\mathcal Z,\varphi_0} 1 + \|U_1(R_1,R_2)\|_{-\varepsilon-\delta} \lesssim_\mathbb{Z} (1+2^{-n\delta}\|R_1 + R_2\|_0)^2.
    \]
    Notice that $n\geq1$ is chosen so that $2^{-n(2-2\varepsilon)}(\|R_1+R_2\|_0\vee1)\simeq 1$. Taking $\delta = 2 - 2\varepsilon$, we obtain the uniform bound $\|R_1\|_0 \lesssim_{\mathcal Z,\varphi_0} 1$. On the other hand, choosing $\delta =0$, we have $\|R_1\|_{2-\varepsilon} \lesssim_{\mathcal Z,\varphi_0} 1 +\|R_2\|_0^2$. The estimate for $R_1$, stated in~\eqref{C4_1}, is then followed by interpolation.
    
    To derive the estimate for $R_2$ , we apply the maximum principle (\cite[Lemma 3.6]{Jagannath2023}), which yields
    \[ \|R_2\|_0 \lesssim_{\mathcal Z,\varphi_0} \|U_2(R_1,R_2)\|_0^{\frac{1}{3}}.\]
    By Lemma~\ref{lem:U} and the estimate for $R_1$ in~\eqref{C4_1}, we obtain 
    \begin{align*}
    \|R_2\|_0 &\lesssim_{\mathcal Z,\varphi_0} 1+ \|R_2\|_0^{\frac{1}{3}} + \|R_2\|_0^{\frac{2}{3}} + \|R_1 + R_2\|_{\varepsilon'}^{\frac{1}{3}}(1+\|R_1 + R_2\|_0^{\frac{1}{3}})\\
    &\quad + 2^\frac{n\kappa}{3}\|R_1 + R_2\|_0^{\frac{1}{3}} + 2^\frac{2n\kappa}{3}\|R_1 + R_2\|_0^{\frac{2}{3}},\\
    &\lesssim_{\mathcal Z,\varphi_0} 1+ \|R_2\|_0^{\frac{1}{3}} + \|R_2\|_0^{\frac{2}{3}} + (\|R_2\|_{\varepsilon'} + \|R_2\|_0^{\frac{2\varepsilon'}{2-\varepsilon}})^{\frac{1}{3}}(1+\|R_2\|_0^{\frac{1}{3}})\\
    &\quad + 2^\frac{n\kappa}{3}\|R_2\|_0^{\frac{1}{3}} + 2^\frac{2n\kappa}{3}\|R_2\|_0^{\frac{2}{3}}.
    \end{align*}
    Applying Young's inequality, we further obtain
    \begin{align*}
    \|R_2\|_0 &\lesssim_{\mathcal Z,\varphi_0} 1+ \|R_2\|_{\varepsilon'}^{\frac{1}{3}}(1+\|R_2\|_0^{\frac{1}{3}}) + 2^\frac{2n\kappa}{3}\|R_2\|_0^{\frac{2}{3}}\\
    &\lesssim_{\mathcal Z,\varphi_0} 1+ \|R_2\|_{\varepsilon'}^{\frac{1}{3}}(1+\|R_2\|_0^{\frac{1}{3}}) + \|R_2\|_0^{\frac{2\kappa}{3(2-2\varepsilon)}+\frac{2}{3}}.
    \end{align*}
    Taking $\kappa \in (0,1-\varepsilon)$ and applying Young's inequality once more, we obtain the desired bound 
    \[ \|R_2\|_0 \lesssim_{\mathcal Z,\varphi_0} 1+ \|R_2\|_{\varepsilon'}^{\frac{1}{2}}\]
    as stated in~\eqref{C4_2}.
    
    To derive the estimate for $\|R_2\|_{2-\varepsilon}$, we apply Schauder estimate, Lemma~\ref{lem:U}, and the bounds~\eqref{C4_1} and~\eqref{C4_2}:
    \begin{align*}
    \|R_2\|_{2-\varepsilon}  &\lesssim_\mathcal Z \|-R_2^3 +U_2\|_{-\varepsilon} \lesssim_{\mathcal Z} \|R_2\|_0^3 +\|U_2\|_{-\varepsilon} \\
    &\lesssim_{\mathcal Z,\varphi_0} 1+ \|R_2\|_0^3 + (\|R_2\|_0^{\frac{2\varepsilon'}{2-\varepsilon}} + \|R_2\|_{\varepsilon'})(1+\|R_2\|_0)\\
    &\lesssim_{\mathcal Z,\varphi_0} 1+ \|R_2\|_0^3 + \|R_2\|_{\varepsilon'}^{\frac{3}{2}}.
    \end{align*}
    Finally, by interpolation,
    \[ \|R_2\|_{\varepsilon'}^{\frac{3}{2}} \lesssim (\|R_2\|_{2-\varepsilon}^{\frac{\varepsilon'}{2-\varepsilon}}\|R_2\|_0^{1-\frac{\varepsilon'}{2-\varepsilon}})^{\frac{3}{2}}, \]
    and applying Young’s inequality yields the bound
    \[ \|R_2\|_{2-\varepsilon} \lesssim_{\mathcal Z,\varphi_0} 1+\|R_2\|_0^3 \]
    as stated in~\eqref{C4_3}.    
\end{proof}

\begin{proof}[Proof of Theorem~\ref{thm:global}]
    By Corollary~\ref{cor:u}, we have the bounds
    \[ \|R_1\|_{2-\varepsilon} \lesssim_{\mathcal Z,\varphi_0} 1+\|R_2\|_0^2 \quad \text{and} \quad \|R_2\|_{2-\varepsilon} \lesssim_{\mathcal Z,\varphi_0} 1+\|R_2\|_0^3. \]
    Therefore, it remains to prove that $\|R_2\|_0 \lesssim_{\mathcal Z,\varphi_0} 1$.
    By interpolation,
    \[ \|R_2\|_{\varepsilon'} \lesssim \|R_2\|_{2-\varepsilon}^{\frac{\varepsilon'}{2-\varepsilon}}\|R_2\|_0^{1-\frac{\varepsilon'}{2-\varepsilon}}. \]
    Thus, 
    \begin{align*}
        \|R_2\|_0 &\lesssim_{\mathcal Z,\varphi_0} 1 + \|R_2\|_{\varepsilon'}^{\frac{1}{2}}\\
        &\lesssim_{\mathcal Z,\varphi_0} 1 + \|R_2\|_{2-\varepsilon}^{\frac{1}{2}\cdot\frac{\varepsilon'}{2-\varepsilon}}\|R_2\|_0^{\frac{1}{2}(1-\frac{\varepsilon'}{2-\varepsilon})}\\
        &\lesssim 1 + \|R_2\|_{2-\varepsilon}^{\frac{\varepsilon'}{2-\varepsilon}} + \|R_2\|_0^{1-\frac{\varepsilon'}{2-\varepsilon}}\\
        &\lesssim_{\mathcal Z,\varphi_0} 1 + \|R_2\|_0^{\frac{3\varepsilon'}{2-\varepsilon}} + \|R_2\|_0^{1-\frac{\varepsilon'}{2-\varepsilon}}.
    \end{align*}
    For $\varepsilon$ and $\varepsilon'$ sufficiently small, Young's inequality implies
    \[ \|R_2\|_0 \lesssim_{\mathcal Z,\varphi_0} 1. \]
    Therefore, we have established the desired bounds for $R_1$ and $R_2$. In particular, we obtain $\|R\|_{2-\varepsilon} \leq C$ for any $T < T^* \wedge 1$, which implies $T^* > 1$. By iterating this argument, we conclude that $T^* = \infty$.
\end{proof}

We can now prove Theorem \ref{thm:continuity}.

\begin{proof}[Proof of Theorem \ref{thm:continuity}]
Theorems~\ref{thm:local} and \ref{thm:global} establish the local and global well-posedness and continuous dependence for the simplified equation \eqref{eqn_u}. The full equation \eqref{eq:R} only differs by the additional linear term $\alpha R^{\varphi_0,\mathcal Z}$, which can be readily included in the Picard iteration and the a priori bounds in Theorems~\ref{thm:local} and \ref{thm:global}. Thus, the continuous dependence established in Theorems~\ref{thm:local} and \ref{thm:global} holds for the remainder map $\mathcal R$, and the solution $R^{\varphi_0,\mathcal Z}$ is bounded in $C_T\cC^{2-\varepsilon}$. Moreover, the projection $(\varphi_0,(Z_1,Z_2,Z_3)) \mapsto Z_1$ is continuous from $\cC^{2-\varepsilon} \times (C_T\cC^{-\varepsilon})^3$ into $C_T \cC^{-\varepsilon}$. Since $\cC^{2-\varepsilon} \hookrightarrow \cC^{-\varepsilon}$, the map $\mathcal R$ is also continuous from $\cC^{-\varepsilon} \times (C_T \cC^{-\varepsilon})^3$ into $C_T\cC^{-\varepsilon}$. Therefore, the map $\Psi_\alpha$ is continuous from $\cC^{-\varepsilon} \times (C_T\cC^{-\varepsilon})^3$ into $C_T \cC^{-\varepsilon}$.
\end{proof}

\subsection{Proof of Lemma~\ref{lem:supp}}\label{a2}

\begin{proof}

If $x \in \tmop{supp}(\mathbb{P}_X)$, then every neighbourhood $U$ of $x$ has $\mathbb{P}_X(U) > 0$. For any neighbourhood $V$ of $\varphi(x)$, the preimage $\varphi^{-1}(V)$ is a neighbourhood of $x$, so $\mathbb{P}_{\varphi(X)}(V) = \mathbb{P}_X(\varphi^{-1}(V)) > 0$. Thus, $\varphi(x) \in \tmop{supp}(\mathbb{P}_{\varphi(X)})$. This shows
\[ \varphi(\tmop{supp}(\mathbb{P}_X)) \subset \tmop{supp}(\mathbb{P}_{\varphi(X)}), \]
and therefore
\[ \overline{\varphi(\tmop{supp}(\mathbb{P}_X))} \subset \tmop{supp}(\mathbb{P}_{\varphi(X)}). \]
For the converse inclusion, let $y \in \tmop{supp}(\mathbb{P}_{\varphi(X)})$. Let  $V_n:=B(y,\tfrac1n)$ be the open ball with radius $\tfrac1n$ around $y$. Since $y \in \tmop{supp}(\mathbb{P}_{\varphi(X)})$, we have $\mathbb{P}_{\varphi(X)}(V_n)>0$ for all $n$. Thus, $\mathbb{P}_X(\varphi^{-1}(V_n)) = \mathbb{P}_{\varphi(X)}(V_n)>0$. By inner regularity of $\mathbb{P}_X$, for each $n$, there exists a compact set $K_n \subset \varphi^{-1}(V_n)$ such that $\mathbb{P}_X(K_n) > 0$. Since $\mathbb{P}_X(K_n) > 0$ and $K_n$ is compact, the support of $\mathbb{P}_X |_{K_n}:= \mathbb{P}_X(\cdot \cap K_n)$ is nonempty and $\tmop{supp}(\mathbb{P}_X |_{K_n}) \subset K_n \subset \varphi^{-1}(V_n)$. Therefore, there exists $x_n \in \tmop{supp}(\mathbb{P}_X |_{K_n}) \subset \varphi^{-1}(V_n)$. Thus, we have $\varphi(x_n) \in V_n \cap \varphi(\tmop{supp}(\mathbb{P}_X))$. As $V_n = B(y,\tfrac1n)$, we have $\varphi(x_n) \to y$, and therefore $y \in \overline{\varphi(\tmop{supp}(\mathbb{P}_X))}$. We have thus shown that \[\tmop{supp}(\mathbb{P}_{\varphi(X)}) \subset \overline{\varphi(\tmop{supp}(\mathbb{P}_X))}.\]
\end{proof}

\bibliographystyle{alpha}
	\bibliography{all}

\end{document}